\documentclass[preprint,12pt]{elsarticle}

\usepackage{amsthm}
\usepackage{amsfonts}
\usepackage{amsmath}
 \usepackage[all]{xy}

\newtheorem{theorem}{Theorem}[section]
\newtheorem{lemma}[theorem]{Lemma}
\newtheorem{proposition}[theorem]{Proposition}

\theoremstyle{definition}
\newtheorem{definition}[theorem]{Definition}
\newtheorem{example}[theorem]{Example}

\theoremstyle{remark}

\theoremstyle{remark}
\newtheorem{remarks}[theorem]{Remarks}

\journal{Journal of Algebra}

\begin{document}

\begin{frontmatter}



\title{On the number of isomorphism classes of quasigroups}

  \author[label1]{P. Jim\'enez-Seral}
   \author[label2]{I. Lizasoain}
   \author[label2]{G.Ochoa}
  \address[label1]{Departamento de Matem\'aticas,\\ Facultad de Ciencias,
Universidad de Zaragoza,\\ Pedro Cerbuna 12, 50009 Zaragoza, Spain\\
paz@unizar.es}
  \address[label2]{ Departamento de Matem\'aticas,\\
Universidad P\'ublica de Navarra,\\ Campus de Arrosad\'{\i}a, 31006 Pamplona, Spain\\
ilizasoain@unavarra.es, ochoa@unavarra.es}
 
\fntext[label3]{The research of the first and the second authors is supported by  
{\em Proyecto MTM2010-19938-C03-01/03} of the {\em Ministerio de Ciencia e Innovaci\'on de Espa\~na.}}

\begin{abstract}
In this paper the number of isomorphism classes of the right transversals of the symmetric group $\Sigma_{n-1}$ in the symmetric group $\Sigma_n$ is calculated. This result together with the techniques employed along the paper may be used to go forward in the study and classification of the subgroups with a certain number of isomorphism classes of transversals in the group.
\end{abstract}


\end{frontmatter}


\section{Introduction}

Throughout this paper, all groups considered are finite.

The classical theory of extensions of groups has focused on aspects which are independent of the choice of transversal of the smaller group in the largest one. For instance, either the representation of the group as a permutation group of its right cosets, or the induced representation of a group from one of its subgroups, does not depend, up to equivalence, on the choice of a right transversal of the subgroup in the whole group (see for instance \cite{Luccini00} or \cite{Luccini01}). 

However, this point of view leads to some lack of accuracy in certain definitions. For this reason some recent papers have analysed the set of all right transversals of a subgroup in a group as an abstract object (\cite{JainShukla08}, \cite{JainShukla11}). Any right transversal of a subgroup in a group has the structure of a right quasigroup with identity. But what really makes this structure important is that any right quasigroup with identity of order $n$ is isomorphic to a right transversal of the symmetric group $\Sigma_{n-1}$ in $\Sigma_n$, as R. Lal showed in the key paper \cite{Lal96}.

The main aim of this paper is to establish a formula to compute the number of all isomorphism classes of right quasigroups with identity of a given order. This result is presented in Theorem \ref{main}.
\section{Preliminaries}
\begin{definition} A non-empty set $S$ endowed with an inner operation $\circ$ is said to be a {\em right quasigroup} if, for any pair $a,b\in S$, there exists a unique $x\in S$ such that $x\circ a=b$. If there exists an element $e$ such that $e\circ a=a=a\circ e$, for any $a\in S$, $S$ will be called a {\em right quasigroup with identity}. 

Given two right quasigroups with identity $S$ and $T$, a bijective map $f:S\to T$ is said to be an {\em isomorphism} if $(x\circ y)^f=x^f\circ y^f$ for any $x,y\in S$. (Here we use the same symbol $\circ$ to denote the inner operation of both $S$ and $T$). In this case we will say that $S$ and $T$ are isomorphic. 
\end{definition}

\begin{remarks}
\label{asociative}
Let $S$ be a right quasigroup with identity.
\begin{enumerate}
\item The element $e$ which appears in the definition is unique and is called the identity element of $S$.
\item If $S$ and $T$ are right quasigroups with identity, $f:S\to T$ is an isomorphism and $e$ is the identity element of $S$, then $e^f$ is the identity element of $T$.
\item Each element $a\in S$ induces a permutation of the elements of $S$ given by $x^{a}=b$ if and only if   $x\circ a=b$, for each $x\in S$.
\item Assume that the operation $\circ$ is associative. For any $a\in S$ there exists a unique element $b\in S$ such that $b\circ a=e$. Note that $(a\circ b)\circ a=a\circ (b\circ a)=a\circ e=a$. Then $a\circ b=e$. Hence, every element of $S$ possesses an inverse in $S$.

Thus, if the operation $\circ$ is associative, $S$ is a group.
\item We say that $S$ has {\em order} $n$ if the cardinal of the set $S$ is $n$.
\end{enumerate}
\end{remarks}

The introduction of right quasigroups with identity, as an algebraic structure, enables us the analysis of the right transversals of a subgroup in a group. If $H$ is a subgroup of a group $G$, $e$ denotes the identity element of $G$ and $S$ is a right transversal of $H$ in $G$ such that $e\in S$, then for any $x,y\in S$, there exists a unique $z\in S$ with $Hxy=Hz$. Therefore the inner operation $\circ$ in $S$ defined by

$$x\circ y=z \Longleftrightarrow Hxy=Hz\,\,\mbox{for any $x, y\in S$},$$
gives an structure of right quasigroup with identity to the set $S$.

In fact, this example is prototypical. 

As usual, we denote the group of all permutations of the set $\{1,2,\ldots,n\}$ by $\Sigma_n$. The subgroup of $\Sigma_n$ consisting of all permutations fixing the figure $n$ is isomorphic to $\Sigma_{n-1}$. This fact determines an embedding of $\Sigma_{n-1}$ in $\Sigma_n$ and, in what follows, we consider $\Sigma_{n-1}$ as a subgroup of $\Sigma_n$. Note that for any $x\in \Sigma_n$, the coset $\Sigma_{n-1}x$ consists of all permutations moving $n$ to $n^x$. Next we prove that any finite right quasigroup with identity is isomorphic to a right transversal of $\Sigma_{n-1}$ in $\Sigma_n$. This fact appears in \cite{Shukla95} (Theorem 3.7) but we include here a direct, and simpler, proof for the sake of completeness. 

\begin{proposition}[\cite{Shukla95} Theorem 3.7]
\label{quasigroup}
 Any right quasigroup with identity of order $n$ is isomorphic to some right transversal of $\Sigma_{n-1}$ in $\Sigma_n$.
\end{proposition}

\begin{proof}
Let $S=\{ x_1,x_2,\ldots ,x_n=e \}$ be a right quasigroup with identity element $e$. For each $1\leq i\leq n$, consider the permutation $\sigma_i\in \Sigma_n$ given by $j^{\sigma_i}=k$ whenever $x_j\circ x_i=x_k$. Notice that, for all $1\leq i\leq n$, $\sigma_i$ is a bijective map with $n^{\sigma_i}=i$. In addition, $\sigma_n$ agrees with the identity permutation of $\{1,2,\ldots,n\}$. 

In order to see that $\{\sigma_1,\ldots,\sigma_n\}$ is a right transversal of $\Sigma_{n-1}$ in $\Sigma_n$, suppose that $\sigma_i=\rho\sigma_j$ for some $\rho\in \Sigma_{n-1}$. Then, $$i=n^{\sigma_i}=n^{\rho\sigma_j}=n^{\sigma_j}=j.$$
We conclude that $\sigma_1,\ldots,\sigma_n$ belong to different right cosets of $\Sigma_{n-1}$ in $\Sigma_n$ and hence $\{\sigma_1,\ldots,\sigma_n\}$ is a complete set of representatives of the $n$ right cosets of $\Sigma_{n-1}$ in $\Sigma_n$.

Notice that, for any $\rho, \rho' \in \Sigma_n$, $\Sigma_{n-1}\rho = \Sigma_{n-1}\rho' \Longleftrightarrow n^{\rho}= n^{\rho'}$. 
Hence, 
$$\Sigma_{n-1}\sigma_i\sigma_j=\Sigma_{n-1}\sigma_k\Longleftrightarrow n^{\sigma_i\sigma_j}=n^{\sigma_k}\Longleftrightarrow
x_i\circ x_j=x_k.$$
Consequently, the map from $S$ to the right transversal $\{\sigma_1,\ldots,\sigma_n\}$ which maps $x_i$ to $\sigma_i$ for each $1\leq i\leq n$, provides an isomorphism between both right quasigroups with identity.
\end{proof}

For any subgroup $H$ of $G$, $\mathcal{T}(G,H)$ will denote the set of all right transversals of $H$ in $G$ containing $e$ considered as right quasigroups with identity. The set of all isomorphism classes of these right quasigroups with identity will be denoted by $\mathcal{I}(G,H)$.

\begin{remarks}
\label{observa}
Let $G$ be a group and $H$ a subgroup of $G$.  Put $\vert G:H\vert=n$.
\begin{enumerate}
\item Assume that $N$ is a normal subgroup of $G$ with $N\leq H$. For any element $g\in G$, let $\overline{g}=Ng$ denote the image of $g$ under the natural epimorphism of $G$ onto $G/N$. Analogously, for any subset $X$ of $G$ we write $\overline{X}=\{\overline{x};\, x\in X\}$. For any right transversal $S\in \mathcal{T}(G,H)$, then $\overline{S}$ is a right transversal of $\overline{H}$ in $\overline{G}$. In addition, for any $x,y\in S$, if $Hxy=Hz$, with $z\in S$, then $xy=hz$ for some $h\in H$ and $\overline{H}\overline{x}\;\overline{y}=\overline{H}\overline{xy}=\overline{H}\,\overline{hz}=\overline{H}\overline{z}$. Hence $S$ and $\overline{S}$ are isomorphic quasigroups with identity. In other words, the natural epimorphism of $G$ onto $\overline{G}$ induces a bijection between $\mathcal{I}(G,H)$ and $\mathcal{I}(\overline{G},\overline{H})$.

\item  If $H$ and $K$ be subgroups of $G$ with $HK=G$, then any right transversal of $H\cap K$ in $K$ is also a right transversal of $H$ in $G$. In addition, $\mathcal{I}(K,H\cap K)\subseteq \mathcal{I}(G,H)$.

\item If $f: G_1\to G_2$ is a group isomorphism and $H_1\leq G_1$, then the restriction of $f$ to any right transversal $S$ of $H_1$ in $G_1$ provides a quasigroup isomorphism between $S$ and the right transversal $S^f$ of $H_1^f$ in $G_2$. 

\item The natural action of $G$ on the set of right cosets of $H$ in $G$ provides a group homomorphism $\varphi$ from $G$ to $\Sigma_n$. In fact, for any right transversal $\{x_1,\ldots,x_n=e\}$ of $H$ in $G$, $i^{\varphi(g)}=k$ if and only if $Hx_ig=Hx_k$. In particular,

$$i^{\varphi(x_j)}=k \Longleftrightarrow Hx_ix_j=Hx_k \Longleftrightarrow x_i\circ x_j=x_k.$$

Notice that $n^{\varphi(x_i)}=i$ for any $1\leq i\leq n$ and $n^{\varphi (h)}=n$ for every $h\in H$. Hence, $\varphi(H)\subseteq \Sigma_{n-1}$. Moreover, $\{\varphi(x_1),\ldots,\varphi(x_n)=\mbox{id}\}$ is a right transversal of $\Sigma_{n-1}$ in $\Sigma_n$.

In addition, if $\{x_1,\ldots,x_n=e\}$ and $\{y_1,\ldots,y_n=e\}$ are right transversals of $H$ in $G$, then they are isomorphic as quasigroups if and only if $\{\varphi(x_1),\ldots,\varphi(x_n)=\mbox{id}\}$ and $\{\varphi(y_1),\ldots,\varphi(y_n)=\mbox{id}\}$ are isomorphic.

\item If $\mbox{Core}_G(H)$, the largest normal subgroup of $G$ contained in $H$, is the identity subgroup, then $\varphi$ is one-to-one and we can consider $G\leq \Sigma_n$ and $H\leq \Sigma_{n-1}$. Therefore any right transversal of $H$ in $G$ is a right transversal of $\Sigma_{n-1}$ in $\Sigma_n$ in this case. 
\end{enumerate}
\end{remarks}

\section{Counting the isomorphism classes of right quasigroups with identity}
For any positive integer $n$, let $QG(n)$ denote the number of isomorphism classes of right quasigroups with identity of order $n$. In this section we present a formula to reckon $QG(n)$. By Proposition \ref{quasigroup} we have to calculate the cardinal of the set $\mathcal{I}(\Sigma_n,\Sigma_{n-1})$. Thus it is necessary to give some notation of permutation groups as it will appear in this section.

Let $t$ denote any conjugacy class of $\Sigma_{n-1}$. The cardinal of $t$ will be denoted by $a_t$ and we write $c_t=\vert C_{\Sigma_{n-1}}(\sigma)\vert$, where $\sigma$ is a representative of the elements in $t$ and $C_{\Sigma_{n-1}}(\sigma)$ denotes the centralizer of $\sigma$ in $\Sigma_{n-1}$. Recall that $a_tc_t=(n-1)!$. It is well-known that each $t$ is characterized by an $(n-1)$-tuple of non-negative integers $(r_{t,1},\ldots,r_{t,n-1})$ with $r_{t,1}+2r_{t,2} \cdots +(n-1) r_{t,n-1}=n-1$, where $r_{t,k}$ is the number of cycles of length $k$ for each $k=1,2,\ldots,n-1$, occurring in any element of $t$ written as a product of disjoint cycles (see \cite{Rotman} Theorems 3.1, 3.2 and 3.10).

Note that next Lemma \ref{isomorphic} is a particular case of a result proved in the paper of R. Lal, \cite{Lal96} Theorem 2.6. We include here a proof written in terms of permutation groups.

\begin{lemma}
\label{isomorphic}
Two right transversals of $\Sigma_{n-1}$ en $\Sigma_n$ are isomorphic (as right quasigroups with identity) if and only if they are conjugate by some $\sigma\in \Sigma_{n-1}$.
\end{lemma}

\begin{proof}
Suppose that $S=\{x_1,\cdots ,x_n=e\}$ and $T=\{y_1,\cdots ,y_n=e\}$ are right transversals of $\Sigma_{n-1}$ in $\Sigma_n$ such that, for any $1\leq i\leq n$, $n^{x_i}=n^{y_i}=i$. Note that if $\tau_i$ denotes the transposition $(i,n)$, then $\Sigma_{n-1}x_i=\Sigma_{n-1}\tau_i=\Sigma_{n-1}y_i$.

Let $f:S\to T$ be a bijective map providing the isomorphism between $S$ and $T$. Define $\sigma\in \Sigma_{n-1}$ by means of $i^{\sigma}=m$ whenever $x_i^{\,\,f}=y_m$. This means $x_i^{\,\,f}=y_{i^{\sigma}}$. For any $1\leq i,j\leq n$, take $k$ such that $x_i\circ x_j=x_k$. Then,
\begin{eqnarray*}
&& (i^{\sigma})^{x_j^{\,\,\sigma }}=(i^{\sigma})^{\sigma ^{-1} x_j\sigma }=i^{x_j\sigma }=n^{(x_i\circ x_j)\sigma}=n^{x_k\sigma}=k^{\sigma}=n^{y_{k^{\sigma}}}= n^{x_k^{\,\,f}}\\
&& = n^{(x_i^{f}\circ x_j^{f})}=n^{(y_{i^{\sigma}}\circ y_{j^{\sigma}})}= (i^{\sigma})^{y_{j^{\sigma }}},
\end{eqnarray*} 
which shows that $x_j^{\,\,\sigma }=y_{j^{\sigma }}$.

For the converse, conjugation by any $\sigma \in \Sigma_{n-1}$ is an automorphism of $\Sigma_n$ which fixes $\Sigma_{n-1}$. Then it provides a quasigroup isomorphism between any right transversal of $\Sigma_{n-1}$ in $\Sigma_n$ and its $\sigma$-conjugate.  
 \end{proof}

\begin{theorem}
\label{main}
The number of isomorphism classes of right quasigroups with identity of order $n$ is

$$QG(n)=\frac{1}{\vert \Sigma_{n-1}\vert} \sum _ta_tc_{t^1}^{\,r_{t,1}}c_{t^2}^{\,r_{t,2}}\cdots c_{t^{n-1}}^{\,r_{t,n-1}}$$
where $t$ runs over the set of all conjugacy classes of $\Sigma_{n-1}$ and, for each $t$, if $\sigma\in t$, we write $t^k$ to denote the conjugacy class of $\sigma^k$, for $k=1,\ldots,n-1$.
\end{theorem}

\begin{proof}
As commented above, by Proposition \ref{quasigroup}, $QG(n)=\vert\mathcal{I}(\Sigma_n,\Sigma_{n-1})\vert$. By Lemma \ref{isomorphic}, $QG(n)$ is the number of orbits of the action by conjugation of the group $\Sigma_{n-1}$ on the set $\mathcal{T}_n=\mathcal{T}(\Sigma_n,\Sigma_{n-1})$. By a formula due to Burnside (see \cite{Biggs}, Theorem 14.4), we have

\begin{eqnarray}
QG(n)= \frac{1}{\vert \Sigma_{n-1} \vert}\sum_{\sigma\in \Sigma_{n-1}}\vert \mbox{fix}(\sigma)\vert 
\end{eqnarray}

where $\mbox{fix}(\sigma)=\{S\in \mathcal{T}_n ;\, S^{\sigma}=S\}$.

Fix $\sigma\in \Sigma_{n-1}$ and consider that $\sigma=\rho_1\ldots \rho_m$ is the decomposition of $\sigma$ as a product of disjoint cycles including the cycles of length one. 
Consider any right transversal $S=\{x_1,\ldots,x_n=e\}$ of $\Sigma_{n-1}$ in $\Sigma_n$. Assume that, for each $i=1,\ldots,n$, $x_i\in \Sigma_{n-1}\tau_i$ where $\tau_i=(i,n)$. Since $n^{x_i^{\,\,\sigma}}=n^{\sigma^{-1} x_i \sigma}=n^{x_i\sigma}=i^{\sigma}$, we have
$x_i^{\,\,\sigma}\in \Sigma_{n-1}x_{i^{\sigma}}$ for any $i=1,\ldots,n-1$. 
Therefore $S^{\sigma}=S$ if and only if $x_i^{\,\,\sigma}=x_{i^{\sigma}}$, for all $i=1,\ldots,n-1$. Note that $i$ and $i^{\sigma}$ belong to the same cycle in the decomposition of $\sigma$. This is to say that if $j$ and $i$ appear in the same cycle of the decomposition of $\sigma$ and $S^{\sigma}=S$, then the element $x_j\in S$ is determined by $x_i$ and $\sigma$. Thus, for each $l\in \{1,\ldots,m\}$ choose a figure $i(l)$ belonging to the cycle $\rho_l$. From the above argument we deduce that in order to get a right transversal $S\in \mbox{fix}(\sigma)$, we only have to give elements $x_{i(1)},\ldots, x_{i(m)}$ since the rest are determined by these elements and $\sigma$. If for each $l\in \{1,\ldots,m\}$, we write $P_{\sigma}(l)$ to denote the number of possible $x_{i(l)}$, then we have
$$\displaystyle \vert \mbox{fix}(\sigma)\vert=\prod_{l=1}^{m} P_{\sigma}(l).$$

Let $l\in \{1,\ldots,m\}$. Write $i$ for $i(l)$ and $k_l$ for the length of the cycle $\rho_l$. Note that $\sigma^{k_l}$ fixes both figures $i$ and $n$.
Now, any representative of the coset $\Sigma_{n-1}\tau_i$ is of the form $x_i=h\tau_i$ for some $h\in \Sigma_{n-1}$. If, in addition, $x_i\in S$ for some $S\in \mbox{fix}(\sigma)$, then $x_i^{\,\,\sigma ^{k_l}}=x_{i^{\sigma^{k_l}}}=x_i$ and consequently $x_i\in C_{\Sigma_n}(\sigma ^{k_l})$. But $\tau_i$ also belongs to $C_{\Sigma_n}(\sigma ^{k_l})$ because $\sigma ^{k_l}$ fixes both $i$ and $n$. Hence, 
$$h=x_i\tau_i^{-1}\in C_{\Sigma_n}(\sigma ^{k_l})\cap \Sigma_{n-1}=C_{\Sigma_{n-1}}(\sigma ^{k_l}).$$

 Conversely, if $h\in  C_{\Sigma_{n-1}}(\sigma ^{k_l})$, then 
 $$\{x_i^{\,\,\sigma^k}=h^{\sigma^k}\tau_{i^{\sigma^k}};\;\;k=0,1,\ldots,k_l-1\}$$
is a set of representatives of the right cosets $\Sigma_{n-1}\tau_{i^{\sigma^k}}$, $k=0,1,\ldots,k_l-1$, which is fixed by $\sigma$.
 
  Hence, writing $t$ to denote the conjugacy class of $\sigma$, we have
  $$P_{\sigma}(l)=\vert C_{\Sigma_{n-1}}(\sigma ^{k_l})\vert =c_{t^{k_l}}$$
  and then
  $$\vert \mbox{fix}(\sigma)\vert=\prod_{l=1}^m c_{t^{k_l}}.$$
  Note that $\vert\mbox{fix}(\sigma)\vert$ depends only on the structure of $\sigma$ as a product of disjoint cycles and two conjugate permutations have the same cycle structure. Thus in formula $(1)$ we can write 
  
$$QG(n)=\frac{1}{\vert \Sigma_{n-1} \vert}\sum_{\sigma\in \Sigma_{n-1}}\vert \mbox{fix}(\sigma)\vert =\frac{1}{\vert \Sigma_{n-1} \vert}\sum_t a_t\vert \mbox{fix}(\sigma_t)\vert$$

where $t$ runs over all conjugacy classes of $\Sigma_{n-1}$ and $\sigma_t$ is a representative of the class $t$. Now, if the members of the class $t$ are characterized by the $(n-1)$-tuple $(r_{t,1},\ldots,r_{t,n-1})$, then
$$\vert \mbox{fix}(\sigma_t)\vert=c_{t^1}^{r_{t,1}}c_{t^2}^{r_{t,2}}\cdots c_{t^{n-1}}^{r_{t,n-1}}.$$

Now, the formula becomes clear.
 \end{proof}
 
 \begin{example}
 Let us calculate $QG(6)$, the number of isomorphism classes of right quasigroups with identity of order $6$. Equivalently we want to calculate the cardinal $\vert {\mathcal I}(\Sigma_6,\Sigma_5)\vert$ of the isomorphism classes of right transversals of $\Sigma_5$ in $\Sigma_6$.
 
 Since there are seven ways to decompose $5$ as a sum of positive integers:
 
 $$1+1+1+1+1=1+1+1+2=1+2+2=1+1+3=1+4=5=2+3,$$
 there are seven conjugacy classes in $\Sigma_5$ which can be denoted as
 $$t_{11111},\,t_{1112},\,t_{122},\,t_{113},\,t_{14},\,t_{5},\,t_{23},$$
 A direct calculation gives the following table:
  
 $$\begin{array}{|c|c|c|c|c|c|} \hline 
 \mbox{class} & \mbox{represent.}  & a_t & c_t & (r_{t,ÁÁ1},ÁÁÁÁÁÁ r_{t,ÁÁ2} ,ÁÁÁÁÁÁ r_{t,ÁÁ3} , ÁÁÁÁÁÁ r_{t,ÁÁ4} ,ÁÁÁÁÁÁ  r_{t,ÁÁ5}) & (c_{t^1},ÁÁÁÁÁÁ c_{t^2} ,ÁÁÁÁÁÁ c_{t^3} , ÁÁÁÁÁÁ c_{t^4} , ÁÁÁÁÁÁ c_{t^5}) \\
 \hline
 t_{11111} & \mbox{id} & 1 & 120 & (5,0,0,0,0)& (ÁÁ120ÁÁÁ,ÁÁÁÁ120ÁÁ,ÁÁÁ120ÁÁ,ÁÁÁ120ÁÁ,ÁÁ120ÁÁ)\\
 \hline
t_{1112} & (12) & 10 & 12 & (3,1,0,0,0) & (12,120,12,120,12) \\
 \hline
 t_{122} &(12)(34) & 15 & 8 &  (1,2,0,0,0) & (8,120,8,120,8) \\
  \hline
 t_{113} & (123) & 20 & 6 & (2,0,1,0,0) & (6,6,120,6,6) \\
 \hline
t_{14} & (1234) & 30 & 4 & (1,0,0,1,0) & (4,8,4,120,4) \\
 \hline
 t_{5} & (12345) & 24 & 5 & (0,0,0,0,1) & (5,5,5,5,120) \\
 \hline
 t_{23} & (12)(345) & 20 & 6 & (0,1,1,0,0) & (6,6,12,6,6) \\
 \hline
 \end{array}$$
 
 \vspace{0,3 cm}

Hence, the number of isomorphism classes of right quasigroups with identity of order $6$ is 
\begin{eqnarray*}
&& QG(6)=\frac{1}{120} \left[120^5+ 10 (12^3 \cdot 120) + 15 (8 \cdot 120^2)+ \right.\\
&& \left. 20(6^2 \cdot 120)+30( 4 \cdot 120) +24\cdot 120+20(6 \cdot 12)\right] =207392556.
\end{eqnarray*}
 \end{example}
 
 \begin{remarks}
 \begin{enumerate}
\item Note that the above formula to calculate $QG(n)$ looks similar to an evaluation of the poynomial associated to the symmetric group $\Sigma_{n-1}$ (see \cite{Biggs}). However, appearances deceive!
\item Any group of order $n$ is a right quasigroup with identity and therefore it is isomorphic, as a right quasigroup with identity, to a right transversal of $\Sigma_{n-1}$ in $\Sigma_n$. Moreover, two groups are isomorphic if and only if they are isomorphic as right quasigroups with identity.
\end{enumerate}
 \end{remarks}
{\bf Acknowledgement:} The authors would like to thank Professor L. M. Ezquerro (Dpto. de Matem\'aticas, Universidad P\'ublica de Navarra) for his very helpful comments and suggestions in improving this paper.

 
\end{document}